\documentclass{amsart}
\usepackage{graphicx}
\usepackage{amsmath}
\usepackage{amsfonts}
\usepackage{amssymb}
\usepackage{amsthm}

\newtheorem{thm}{Theorem}[section]

\newtheorem{lem}[thm]{Lemma}

\theoremstyle{definition}
\newtheorem{defn}[thm]{Definition}
\theoremstyle{remark}
\newtheorem{rem}[thm]{Remark}
\numberwithin{equation}{section}
\newcommand{\norm}[1]{\left\Vert#1\right\Vert}

\newcommand{\abs}[1]{\left\vert#1\right\vert}

\newcommand{\R}{\mathbb R}
\newcommand{\C}{\mathbb C}

\def\rp#1{\frac 1 {#1}}
\def\normQR#1#2#3{\norm{#1}_{L_t^{#2}L_x^{#3}}}
\def\normP#1#2{\norm{#1}_{L^{#2}_x}}
\def\normH#1#2{\norm{#1}_{\mathcal{H}^{#2}}}
\def\normHt#1#2{\norm{#1}_{{H}^{#2}}}

\def\Int{\int_{-\infty}^{\infty} }
\def\inT{\int_{0}^{t}}

\def\pair{(q,r)} \def\pairt{(\tilde q, \tilde r)}

\begin{document}

\title{Inhomogeneous Strichartz estimates with spherical symmetry and applications to the Dirac-Klein-Gordon system in two space dimensions.}

\author{Evgeni Ovcharov}
\address{University of Edinburgh, School of Mathematics, JCMB, King's Buildings, Edinburgh EH9 3JZ, UK}
\email{E.Y.Ovcharov@sms.ed.ac.uk}


\date{30 March 2009}
\begin{abstract} In this {\bf draft version} we prove inhomogeneous Strichartz estimates with spherical symmetry in the abstract setting via duality arguments. Then we derive some new explicit estimates in the context of the wave equation. This allows us to prove global existence of spherically symmetric solutions to the Dirac-Klein-Gordon (DKG) system in two space dimensions.
\end{abstract}

\maketitle
\pagestyle{myheadings}
\markboth{EVGENI OVCHAROV}{STRICHARTZ ESTIMATES WITH SPHERICAL SYMMETRY AND THE DKG SYSTEM}

\section{Introduction}
The main motivation behind this paper is the question of the global well-posedness of the DKG system in dimensions higher than one. This is a relativistic field model that describes nuclear interactions of subatomic particles and plays an important role in the relativistic quantum electrodynamics, see \cite{BD}. The system generates a significant mathematical interest too. Mathematically, its main feature is: a system for two quantities where there is an a priori bound for only one of the two in the $L^2$-class and no positive definite energy, but at the same time a presence of a special null-form structure in both nonlinearities allowing the system to be studied at very low regularities, see \cite{GS}, \cite{DFS} \cite{Se2}, \cite{Pe}, \cite{Ma}, \cite{DFS2}.

Strichartz estimates with spherical symmetry have attracted a lot of interest recently. The gain of regularity of these estimates over the standard Strichartz estimates varies with the equation but, for example, in the context of the wave equation this gain is significant. Most attention has been dedicated to the homogeneous setting, see Sterbenz \cite{Sz}, Fang and Wang \cite{FW}, Hidano and Kurokawa \cite{HK}, Machihara et al \cite{MNNO}, Tao \cite{T2}, and Vilela \cite{V2}, with only a few special inhomogeneous estimates being proved. Below we address this issue by making use of duality arguments to show that every homogeneous Strichartz estimate with spherical symmetry has its dual counterpart which is an inhomogeneous Strichartz estimate with spherical symmetry much in the same way as with the standard Strichartz estimates.

The main idea of our proof is to identify the class of $L^p$-functions on $\R^n$ with spherical symmetry with the class of $L^p$-functions on $[0,\infty)$ with the weighted measure $\rho^{n-1}d\rho$. This allows us to use duality and to proceed to a large extent as in the standard case.

Once we get our Strichartz estimates the main challenge will be to come up with the correct definition of spherical symmetry for spinors. It is well-known that the Dirac operator does not preserve spherical symmetry, at least not in the way one expects if one takes the erroneous attitude to treat spinors as normal functions. Thus, we investigate the action of rotations on spinor-space and 
define spherical symmetry for spinors to be the invariance with respect to that action. However, we do not know whether this definition has been used in the physics literature before.

\section{Preliminaries}

We shall make use of the following two results in the sequel.

\begin{lem}[Christ-Kiselev, see lemma 3.1 of \cite{T2}, or \cite{T}] \label{lem: Christ-Kiselev}
Suppose that the integral operator
\begin{equation} \label{eq: Boch int}
  T[F](t)=\int_{-\infty}^{\infty} K(t,s)F(s)ds
\end{equation}
is bounded from $L^p(\R;\mathcal B_1)$ to $L^q(\R;\mathcal B_2)$ for some Banach spaces ${\mathcal B}_1$, ${\mathcal B}_2$ and $1\leq p< q\leq \infty$. The operator-valued kernel $K(t,s)$ maps ${\mathcal B}_1$ to ${\mathcal B}_2$ for all $t,s \in \R$. Assume also that $K$ is regular enough to ensure that (\ref{eq: Boch int}) makes sense as a ${\mathcal B}_2$-valued Bochner integral for almost all $t \in \R$. Then the operator
\[
\tilde{T}[F](t)=\int_{-\infty}^{t} K(t,s)F(s)ds
\]
is also bounded on the same spaces.
\end{lem}

\begin{thm}[D'Ancona, Foschi, Selberg \cite{DFS}] \label{thm: DKG}
Consider the IVP for the DKG system \eqref{eq: DKG 1}, \eqref{eq: DKG 2}  
for initial data in the class  $\psi_{|t=0} =\psi_0 \in L^2$, $\phi_{|t=0}=\phi_0 \in H^r$ and  $\partial_t \phi_{|t=0} =\phi_1 \in H^{r-1}$, where $1/4 < r< 3/4$. Then there exist a time $T>0$, depending continuously on the $L^2 \times H^r \times H^{r-1}$-norm of the data, and a solution
\[
  \psi \in C([0, T]; H^s), \quad  \phi \in C([0, T ]; H^r) \cap C^1([0, T]; H^{r-1}),
\]
of the DKG system \eqref{eq: DKG 1}, \eqref{eq: DKG 2}  on $(0, T ) \times \R^2$, satisfying the initial condition above. Moreover, the solution is unique in this class, and depends continuously on the data.
\end{thm}

\section{Inhomogeneous Strichartz estimates with spherical symmetry}

To every spherically symmetric function $f(x) \in L^p(\R^n)$ we map a function $f_\rho(\rho) \in L^p([0, \infty); \rho^{n-1}d\rho)$ by the rule $f_\rho(\rho) = f(\rho, 0,\dots,0)$. This mapping is a one-to-one isometry. The inverse mapping is defined by the rule $g_x(x)=g(\abs{x})$ where $g(\rho) \in L^p([0, \infty); \rho^{n-1}d\rho)$ and obviously we have $(g_x)_\rho=g(\rho)$ and $(f_\rho)_x=f(x)$. Note that the dual space to $L^p([0, \infty); \rho^{n-1}d\rho)$ is the space $L^{p'}([0, \infty); \rho^{n-1}d\rho)$, where $1\leq p < \infty$, and $p$ and $p'$ are H\"older conjugate.

Suppose now that $\mathcal{H}(\R^n)$ is a Hilbert space of functions on $\R^n$ on which space rotations act as unitary operators. Examples of such include the Sobolev spaces $H^s(\R^n)$ for any $s\in \R$. The class of spherically symmetric functions in $\mathcal{H}(\R^n)$ is a Hilbert space, too, which we shall identify with the Hilbert space  $\mathcal{H}_\rho([0,\infty))$ of functions on $[0, \infty)$ with a scalar product
\[
\langle f,g \rangle_{\mathcal{H}_\rho([0,\infty))} = \langle f_x,g_x \rangle_{\mathcal{H}(\R^n)}.
\]

\begin{lem} Suppose that a linear continuous operator $U(t): \mathcal{H}(\R^n) \rightarrow L^2(\R^n)$ commutes with rotations, i.e. U(t)[f(Rx)]=U(t)[f](Rx), where $R$ denotes a space rotation on $\R^n$. Then its dual $U^*(t): L^2(\R^n) \rightarrow \mathcal{H}(\R^n)$ does too.
\end{lem}

\begin{defn} Let $U(t): \mathcal{H}(\R^n) \rightarrow L^2(\R^n)$ be a linear continuous operator that commutes with space rotations. We define the linear continuous operator $U_\rho(t): \mathcal{H_\rho}([0,\infty)) \rightarrow L^2([0,\infty);\rho^{n-1}d\rho)$ by the rule $U_\rho(t)f=(U(t)f_x)\rho$ for every $f \in \mathcal{H_\rho}([0,\infty))$.
\end{defn}

\begin{lem}
Let $U(t): \mathcal{H}(\R^n) \rightarrow L^2(\R^n)$ be a linear continuous operator and let
$U^*(t): L^2(\R^n) \rightarrow \mathcal{H}(\R^n) $ be its dual. Then the dual to $U_\rho(t): \mathcal{H_\rho}([0,\infty)) \rightarrow L^2([0,\infty);\rho^{n-1}d\rho)$ is the operator
$(U^*(t))_\rho: \mathcal{H_\rho}([0,\infty)) \rightarrow L^2([0,\infty);\rho^{n-1}d\rho)$
\end{lem}
\begin{proof}
\begin{equation}
 \begin{split}
  \langle U_{\rho}(t)f, g \rangle_{L^2([0,\infty);\rho^{n-1}d\rho)} = \rp {\omega_n}
  \langle U(t)f_x, g_x \rangle_{L^2(\R^n)}=\\
  \rp {\omega_n} \langle f_x, U^*(t)g_x \rangle_{\mathcal{H}(\R^n)}=
  \langle f, (U^*(t)g)_{\rho} \rangle_{\mathcal{H_\rho}([0,\infty)},
 \end{split}
\end{equation}
where $\omega_n$ is the area of the unit sphere in $\R^n$.
\end{proof}

Suppose that we have the following estimates for $U(t)$
\begin{align}\label{est: hom}
  \normQR{U(t)f}{q}{r} \lesssim \normH{f}{s}
\end{align}
for all $f \in \mathcal{H}^s$ whenever $(q, r) \in A$, where $\mathcal{H}^s$, $s=s(q,r)$, are a collection of Hilbert spaces of functions on $\R^n$, $f \in \mathcal{H}^s$ is spherically symmetric, and $A$ is the index set of admissability for the estimate \eqref{est: hom}. We can express this more succinctly by saying that the operator
\[
 T:  \mathcal{H}_\rho^s \rightarrow L_t^qL_\rho^r, \qquad Tf=U_\rho(t)f,
\]
where $L_t^qL_\rho^r = L^q((0, \infty); L^r([0, \infty); \rho^{n-1}d\rho))$, is bounded whenever $(q, r) \in A$. Then by duality the operator
\[
  TT^*: L_t^{\tilde q'}L_\rho^{\tilde r'} \rightarrow L_t^qL_\rho^r, \qquad TT^*F=U_\rho(t)\Int U_\rho^*(s)F(s)ds
\]
is bounded too whenever $(q,r), (\tilde q, \tilde r) \in A$. Consider now the operator
\begin{align}\label{oper: W}
    W_\rho(t)F = U_\rho(t)\inT U_\rho^*(s)F(s)ds,
\end{align}
which due to Duhamel's formula expresses the solution to an inhomogeneous PDE whenever $U(t)$ is the linear continuous group associated with that equation and $F(t)$ is a spherically symmetric function with respect to the space variables.

\begin{thm}[Inhomogeneous Strichartz estimates with spherical symmetry] \label{thm: inhom abs}
Suppose that the homogenous Strichartz estimate \eqref{est: hom} holds for all spherically symmetric $f \in \mathcal{H}^s$ whenever $(q, r) \in A$. Then we have that the following inhomogeneous Strichartz estimate
\begin{align}\label{est: inhom}
    \normQR{U(t)\inT U^*(s)F(s)ds}{q}{r}  \lesssim \normQR{F}{\tilde q'}{\tilde r'}
\end{align}
holds for all spherically symmetric $F \in L_t^{\tilde q'}L_x^{\tilde r'}$, whenever  $(q,r), (\tilde q, \tilde r) \in A$ and $q > \tilde q'$ or $(q,r)=(\tilde q, \tilde r)$.
\end{thm}

\begin{proof}
In the case when $q> \tilde q'$ we apply the Christ-Kiselev lemma to the $TT^*$-operator, otherwise we use symmetry considerations as in Keel and Tao \cite{KT}.
\end{proof}

\section{Strichartz estimates for the wave equation}

Define the operators
\begin{align*}
 \widehat {U_{\pm}(t)f} = e^{\pm i(t\abs{\xi})}\hat{f}(\xi), \\
 U_0(t)f = (U_{+}(t-s)+ U_{-}(t-s))/2,\\
 W_0(t)F = \inT \frac {U_{+}(t-s)- U_{-}(t-s)} {2iD} F(s) ds,  
\end{align*}
where the operator $D$ has a Fourier symbol $\abs{\xi}$. Note that $D$ commutes with rotations and thus preserves spherical symmetry.

Then the solution to the IVP for the wave equation
\begin{align}
  \Box u = F(t,x), \quad t \in [0, \infty) \times \R^n, \label{eq: wave 1}\\
  u(0)=f, \quad \partial_t u(0) = g. \label{eq: wave 2}
\end{align}
is given by the formula
\[
 u(t) = \partial_t U_0(t)f + U_0(t)g  + W_0(t)F.
\]
For simplicity, we denote by $U_0(t)[f,g]=\partial_t U_0(t)f + U_0(t)g$ the propagation of the free wave with initial data $f$ and $g$.

\begin{defn}
 We say that the exponent pair $\pair$ is radially wave-admissible if
 \begin{equation} \label{eq: defn rad Strich}
  \rp q + \frac {n-1} r < \frac {n-1} 2, \qquad n >1,
 \end{equation}
where $2 \leq q,r \leq \infty$, $\pair \neq (\infty, \infty)$, or if $\pair$ coincides with $(\infty, 2)$.
\end{defn}

\begin{thm}[\cite{FW}, \cite{HK}] \label{thm: wave hom}
The following estimate
\begin{equation} \label{est: rad Strich}
  \normQR{U_0(t)[f,g]}{q}{r} \lesssim \normHt{f}{s} + \normHt{g}{s-1},
 \end{equation}
holds for all spherically symmetric $f \in \dot{H}^s(\R^n)$, $g \in \dot{H}^{s-1}(\R^n)$, whenever the exponent pair $\pair$ is radially wave-admissible and the Sobolev exponent $s$ satisfies the scaling condition
 \[
  \rp q + \frac n r = \frac n 2 - s.
 \]
\end{thm}

\begin{thm} \label{thm: wave full}
Let $u(t)$ be the solution to the IVP for the wave equation \eqref{eq: wave 1}, \eqref{eq: wave 2}, where $f$, $g$, and $F(t)$ are spherically symmetric. Then the following estimate
\begin{equation*}
  \normQR{D^{\sigma_1}u(t)}{q}{r} \lesssim \norm{f}_{\dot H^s} + \norm{g}_{\dot H^{s-1}} + \normQR{D^{\sigma_2}F}{\tilde q'}{\tilde r'}
\end{equation*}
holds for all $f \in \dot{H}^s(\R^n)$, $g \in \dot{H}^{s-1}(\R^n)$, and $D^{\sigma_2}F(t) \in L^{\tilde q'}_tL^{\tilde r'}_x$ whenever $\pair$, $\pairt$ are two radially wave admissible pairs\footnote{Except when $q=\tilde q=2$, $r\neq \tilde r$, and either $r$ or $\tilde r$ is equal to $\infty$.} and satisfy the following scaling condition
\begin{align}\label{eq: dim cond}
 \rp q + \frac n r - \sigma_1= \frac n 2 - s = \rp {\tilde q'} + \frac n {\tilde r'} - 2-\sigma_2.
\end{align}
\end{thm}
\begin{proof}
The homogeneous Strichartz estimates of theorem \ref{thm: wave hom} hold for each of the operators $U_{\pm}$ separately. For simplicity let us consider $U_{-}(t)$ first. For $U_{-}(t): H^s(\R^n) \rightarrow L^2(\R^n)$ we have that $U_{-}^*(t): L^2(\R^n) \rightarrow H^s(\R^n)$ and $U_{-}^*(t)=D^{-2s}U_{+}(t)$. In view of theorem \ref{thm: wave hom}, the operators $T_1: H^s(\R^n) \rightarrow L^q_t L_x^r$, $T_1f = D^{\sigma_1}U_{-}(t)f$, and $T_2: H^s(\R^n) \rightarrow L^{\tilde q}_tL_x^{\tilde r}$, $T_2f = D^{s-\beta}U_{-}(t)f$ are bounded on spherically symmetric data $f \in H^s(\R^n)$, where
\[
 s= \frac n 2 - \frac n r - \rp q + \sigma_1, \quad \beta = \frac n 2 - \frac n {\tilde r} - \rp {\tilde q},
\]
and $\pair$, $\pairt$ are two radially wave admissible pairs and $q>\tilde q'$. Hence, in view of theorem \ref{thm: inhom abs}, we obtain the estimate
\begin{align*}
  \normQR{\inT U_{-}(t-s)D^{s-\beta-2s}F(s)ds}{q}{r} \lesssim \normQR{F}{\tilde q'}{\tilde r'}.
\end{align*}
Repeating the same argument for $U_{+}(t)$, we obtain the estimate
\begin{align*}
  \normQR{\inT W_0(t-s)F(s)ds}{q}{r} \lesssim \normQR{D^{s+\beta-1}F}{\tilde q'}{\tilde r'}.
\end{align*}
Setting $\sigma_2 = s+\beta-1$ gives condition \eqref{eq: dim cond}.

The case when $\pair$, $\pairt$ are two radially wave admissible pairs with $q=\tilde q=2$ and $r=\tilde r$ is treated similarly.

And finally, the case when $\pair$, $\pairt$ are two radially wave admissible pairs with $q=\tilde q=2$ and $r \neq \tilde r$ is reduced to the previous one by Sobolev embedding.
\end{proof}

\section{Applications}

The two-dimensional DKG system reads
\begin{align}
 (\partial_t + \sigma_1 \partial_x + \sigma_2 \partial_y + iM\sigma_3) \psi(t,x,y) &=i\phi\sigma_3\psi, \quad (t,x,y) \in [0,\infty)\times\R\times\R, \label{eq: DKG 1}\\
 (\partial_t^2-\Delta+m^2)\phi(t,x,y)                                      &=\langle \sigma_3\psi, \psi\rangle, \label{eq: DKG 2}
\end{align}
where
\begin{align}
 \sigma_1 = \begin{pmatrix} 0 & 1 \\ 1 & 0 \end{pmatrix},\quad \sigma_2 = \begin{pmatrix} 0 & -i \\ i & 0 \end{pmatrix},\quad \sigma_3 = \begin{pmatrix} 1 & 0 \\ 0 & -1 \end{pmatrix},
\end{align}
are the Pauli spin matrices and $M$ and $m$ are nonnegative constants. The unknown quantities are a two-spinor $\psi(t,x,y): [0,\infty)\times\R^2 \rightarrow \C^2$, and a real scalar field $\phi(t,x,y): [0,\infty)\times\R^2 \rightarrow \R$.

Let us  recall that the system \eqref{eq: DKG 1}, \eqref{eq: DKG 2} is form covariant with respect to Lorentzian transformations and in particular to spatial rotations. Suppose that the coordinate system $Oxy$ is changed into $Ox'y'$ by a spatial rotation $R(\varphi)$ of an angle $\varphi$
\begin{align*}
 \begin{pmatrix} x' \\ y' \end{pmatrix} = \begin{pmatrix} \cos \varphi & -\sin \varphi \\ \sin \varphi & \cos \varphi \end{pmatrix} \begin{pmatrix} x \\ y \end{pmatrix}.
\end{align*}
Then we want to find a rule $\psi \rightarrow \psi'$ as $Oxy \rightarrow Ox'y'$ of the form $\psi'(t,z')=S(\varphi) \psi(t,z)$, where
$S(\varphi)$ is a $2\times2$ matrix and $z$ denotes $(x,y)$, that leaves \eqref{eq: DKG 1}, \eqref{eq: DKG 2} form invariant. Of course, for the scalar field $\phi$ we have $\phi'(t,z')=\phi(t,z)$. Substituting in  \eqref{eq: DKG 1}, \eqref{eq: DKG 2}
\begin{align*}
 \psi(t,z) = S^{-1}(\varphi)\psi'(t, R(\varphi)z),\\
 \phi(t,z) = \phi'(t,R(\varphi)z),
\end{align*}
we obtain
\begin{align}
 (\partial_t + \sigma_1' \partial_{x'} + \sigma_2' \partial_{y'} + iM\sigma_3') \psi'(t, z') &=i\phi\sigma_3'\psi',  \label{eq: DKG 1'}\\
 (\partial_t^2-\Delta+m^2)\phi'(t,z')                                      &=\langle \sigma_3'\psi', \psi'\rangle, \label{eq: DKG 2'}
\end{align}
where
\begin{align*}
 \sigma_1' &= S(\varphi) \left(  \sigma_1 \cos \varphi  -  \sigma_2 \sin \varphi  \right) S^{-1}(\varphi) \\
 \sigma_2' &= S(\varphi) \left( -\sigma_1  \sin \varphi +  \sigma_2 \cos \varphi  \right) S^{-1}(\varphi)\\
 \sigma_3' &=\ S(\varphi) \sigma_3  S^{-1}(\varphi).
\end{align*}
Thus the matrix $S(\varphi)$ must be such that $\sigma'_j=\sigma_j$, for $j=1,2,3$. One can check that if we set
\[
 S(\varphi) = \begin{pmatrix} e^{i\varphi} & 0 \\ 0 & 1 \end{pmatrix}
\]
all of the above conditions are satisfied. Note that the Klein-Gordan part of the system is form invariant as $\langle \sigma_3'\psi', \psi'\rangle = \langle \sigma_3\psi, \psi\rangle$ due to the fact that $S(\varphi)$ is unitary and the well-known invariance of the Laplacian $\Delta$ with respect to rotations. Thus we come with the following
\begin{defn} We say that the two-spinor $\psi_0(z): \R^2 \rightarrow \C^2$ is spherically symmetric if it satisfies
\begin{align} \label{eq: spin sym}
 \psi_0(R(\varphi)z) = S(\varphi)\psi_0(z).
\end{align}
\end{defn}

\begin{lem} A function $\psi_0(z): \R^2 \rightarrow \C^2$ satisfies \eqref{eq: spin sym} if and only if it has the form
\[
 \psi_0(z) = S(\varphi)\chi(\abs{z}),
\]
where $\varphi$ is the argument of the complex number $x+iy$ and $\chi(\rho) : [0,\infty) \rightarrow \C^2$.
\end{lem}
\begin{proof}
Trivial.
\end{proof}

\begin{rem} From the explicit representation above and the fact that $e^{i\varphi}=(x+iy)/ \abs{z} \in C^\infty(\R^2 \setminus O)$, we see that the smoothness of $\psi_0$ depends on the smoothness of $\chi$ and the behavior of $\chi$ around the origin.
\end{rem}

\begin{lem} Suppose that IVP for \eqref{eq: DKG 1}, \eqref{eq: DKG 2} has a unique solution in some class of initial data. Then for a spherically symmetric data from that class the solution to \eqref{eq: DKG 1}, \eqref{eq: DKG 2} remains spherically symmetric for all time.
\end{lem}
\begin{proof}
Trivial.
\end{proof}

\begin{lem} \label{lem: DKG}
Suppose that $u(t)$ is the solution to the IVP for the wave equation (\ref{eq: wave 1}), (\ref{eq: wave 2}) in space dimension $n=2$. Suppose that the data $f$ and $g$ and the forcing term $F(t)$ are spherically symmetric with $f \in  H^s(\R^2)$, $g \in  H^{s-1}(\R^2)$, and $F(t) \in L^\infty_tL^1_x(\R^2)$. Then we have the estimate
\begin{equation} \label{est: L1}
 \normQR{D^s u(t)}{\infty}{2} \lesssim_T \norm{f}_{\dot H^s(\R^2)} + \norm{g}_{\dot H^{s-1}(\R^2)} + \norm{F}_{L^{\tilde q'}_t([0,T]; L^{1}_x)},
\end{equation}
for $s \in [0, 1/2)$ and $1/{\tilde q}=s$.
\end{lem}
\begin{proof}
We apply theorem \ref{thm: wave full} with $\pair=(\infty, 2)$, $\pairt=(\tilde q, \infty)$, $\tilde q > 2$, $\sigma_1=s$, and $\sigma_2=0$.
\end{proof}

\begin{thm} \label{thm: DKG sph}
Consider the IVP for the DKG system \eqref{eq: DKG 1}, \eqref{eq: DKG 2}, with $m=0$,
for initial data in the class  $\psi_{|t=0} =\psi_0 \in L^2$, $\phi_{|t=0}=\phi_0 \in H^r$ and  $\partial_t \phi_{|t=0} =\phi_1 \in H^{r-1}$, where $1/4 < r< 1/2$ and $\psi_0$, $\phi_0$, and $\phi_1$ are spherically symmetric. Then there exist a spherically symmetric solution
\[
  \psi \in C((0, \infty); L^2), \quad  \phi \in C((0, \infty); H^r) \cap C^1((0, \infty); H^{r-1}),
\]
of the DKG system \eqref{eq: DKG 1}, \eqref{eq: DKG 2}  on $(0, \infty) \times \R^2$, satisfying the initial condition above. Moreover, the solution is unique in this class, and depends continuously on the data.
\end{thm}
\begin{proof}
The fundamental conserved property of the system is the charge estimate
\[
 \normP{\psi(t)}{2} = \normP{\psi_0}{2}.
\]
Using this, the proof follows by standard arguments from theorem \ref{thm: DKG} and lemma \ref{lem: DKG}.
\end{proof}

\nocite{Bou2}
\bibliographystyle{plain}
\bibliography{Dirac}
\end{document}